\documentclass[12pt]{article}
\usepackage{textcomp}
\usepackage{amsmath, amssymb}
\usepackage[utf8]{inputenc}
\usepackage[english]{babel}
\usepackage{amsthm}
\usepackage{graphicx}
\usepackage{float}
\usepackage{setspace}
\usepackage{appendix}
\usepackage{hyperref}
\restylefloat{table}
\usepackage{geometry}
\newtheorem{theorem}{Theorem}[section]
\newtheorem{problem}[theorem]{Problem}
\newtheorem{corollary}[theorem]{Corollary}
\newtheorem{proposition}[theorem]{Proposition}
\newtheorem{lemma}[theorem]{Lemma}
\newtheorem{algorithm}[theorem]{Algorithm}
\newtheorem{protocol}[theorem]{Protocol}
\newtheorem{remark}[theorem]{Remark}
\newtheorem{example}[theorem]{Example}
\usepackage{tikz}
\usetikzlibrary{positioning,chains,fit,shapes,calc}
\usetikzlibrary{decorations.pathreplacing}
\usetikzlibrary{arrows,positioning}
\usepackage{amsfonts}
\usepackage{amssymb}
\usepackage{mathtools}

\def\crit{{\mathcal G}^c}

\def\Nat{{\mathbb N}}

\def\R{{\mathbb R}}

\def\Z{{\mathbb Z}}

\def\Rmax{\R_{\max}}
\def\Zmax{\Z_{\max}}

\def\cycle{Z}

\def\digr{{\mathcal G}}

\def\rem{\operatorname{rem}}

\newcommand{\walkslen}[3]{\mathcal{W}^{#3}(#1\to #2)}
\newcommand{\walks}[2]{\mathcal{W}(#1\to #2)}

\newcommand{\walkslennode}[4]{\mathcal{W}^{#3}(#1\xrightarrow{#4} #2)}

\def\subcrit{{\mathcal G}}

\def\0{-\infty}

\begin{document}
\makeatletter
\title{On the tropical discrete logarithm problem and security of a protocol based on  tropical semidirect product}

\author{Any Muanalifah and Serge\u{\i} Sergeev\footnote{University of Birmingham, School of Mathematics, Edgbaston B15 2TT, UK.} \footnote{Corresponding Author: Serge{\u{\i}} Sergeev. E-mails: any.math13@gmail.com,  sergiej@gmail.com}}

\date{}

\maketitle
\makeatother

\begin{abstract}
Tropical linear algebra has been recently put forward by Grigoriev and Shpilrain
~\cite{grigoriev2014tropical,grigoriev2018tropical} as a promising platform for implementation of protocols of Diffie-Hellman and Stickel type. Based on the CSR expansion of tropical matrix powers, we suggest a simple algorithm for the following tropical discrete logarithm problem: ``Given that $A=V\otimes F^{\otimes t}$ for a unique $t$ and matrices $A$, $V$, $F$ of appropriate dimensions, find this $t$.'' We then use this algorithm to suggest a simple attack on a protocol based on the tropical semidirect product. The algorithm and the attack are guaranteed to work in some important special cases and are shown to be efficient in our numerical experiments.


\end{abstract}

\section{Introduction}
Tropical (max-plus) semiring is the set of real numbers with adjoined negative infinity $\Rmax=\R\cup\{-\infty\}$, equipped with tropical addition $a\oplus b:=\max(a,b)$ and tropical multiplication $a\otimes b:=a+b$. All the usual axioms hold (such as associativity, commutativity and distributivity), however there is a lack of genuine additive inverses: although the definition of $\ominus$ is possible via symmetrization (Baccelli et al.~\cite{baccelli1992synchronization}), it is not straightforward and not easy to use. Instead of this, we have idempotency ($a\oplus a=a$), nonnegativity ($a\geq -\infty$, with $-\infty$ playing the role of additive zero) and close connection to the order: $a\oplus b=b\Leftrightarrow a\leq b$. Note that the multiplicative inverses in $\Rmax$ are well-defined for all elements except for $-\infty$: we have $a^-:=-a$. 

The semiring operations are easily extended to matrices and vectors: we have $(A\oplus B)_{ij}=a_{ij}\oplus b_{ij}$ for any two matrices $A=(a_{ij})$ and $B=(b_{ij})$ with entries in $\Rmax$ of same dimensions, and $(A\otimes B)_{ij}=\bigoplus_k a_{ik}\otimes b_{kj} $ for any two matrices $A$ and $B$ of appropriate dimensions. Using this product, we can also define the tropical matrix powers $A^{\otimes t}=\underset{t}{\underbrace{A\otimes A\otimes \ldots \otimes A}}.$ We can define scalar tropical matrix powers not only for integer exponents but also for real exponents, as follows:
\[a^{\otimes k}= k\times a,\quad\forall a\in\mathbb{R}\quad \forall k\in\mathbb{R}; \quad (-\infty)^{\otimes k}=-\infty,\quad\forall k>0; \quad (-\infty)^{\otimes 0}=0. \]
Note that the $k$th power of $a\in\mathbb{R}$ in the semiring sense is $a$ multiplied by $k$ in the usual sense.

The behaviour of tropical matrix powers is in many ways similar to that of the nonnegative matrix powers (recall that any element of the tropical semiring is nonnegative) and can be considered as tropical counterpart of the classical Perron-Frobenius theory. To this end, we also have the tropical spectral problem, and a theorem that any matrix $F\in\Rmax^{d\times d}$ has at least one tropical eigenvalue, meaning $\lambda\in\Rmax$ such that there is at least one $x\in\Rmax^d$ with at least one component in $\R$ such that $F\otimes x=\lambda\otimes x$. This claim was originally proved by 
Vorobyev~\cite{vorobyev1,vorobyev2}, see also Butkovi\v{c}~\cite{butkovic} for a complete solution of this problem in all cases. Vector $x\in\Rmax^d$ with at least one component in $\R$ such that $F\otimes x=\lambda\otimes  x$ is called a right eigenvector of $F$ and vector $y\in\Rmax^d$ with at least one component in $\R$ such that $y\otimes F=\lambda\otimes y$ is called a left eigenvector of $F$. As usual, left eigenvectors of $F$ are transposed right eigenvectors of $F^T$, and there is no need for a separate theory for them.

The largest tropical eigenvalue of $F\in\Rmax^{d\times d}$, denoted by $\lambda(F)$, can be computed explicitly as follows:
\begin{equation}
\label{e:mcm}
\lambda(F)=\bigoplus_{k=1}^d\bigoplus_{i_1,\ldots, i_k} (F_{i_1i_2}\otimes\ldots\otimes F_{i_ki_1})^{\otimes 1/k}=\max\limits_{1\leq k\leq d}\max\limits_{i_1,\ldots, i_k} 
\frac{F_{i_1i_2}+\ldots+F_{i_ki_1}}{k}.
\end{equation}
Observe that $\lambda(F)=\lambda(F^T)$. Formula~\eqref{e:mcm} is best understood in terms of the associated weighted digraph $\digr(F)=(N,E)$, where $N=\{1,\ldots, d\}$ is the set of nodes and $E=\{(i,j)\colon F_{ij}\neq -\infty\}$ is the set of arcs, weighted by the corresponding entries $F_{ij}$. 
In terms of this digraph, $\lambda(F)$ is the maximum mean weight of all cycles on $\digr(F)$, also called the maximum cycle mean of $\digr(F)$ (or of $F$). When $\digr(F)$ is strongly connected we say that $F$ is irreducible. In this case $\lambda(F)$ is the unique tropical eigenvalue of $F$.

This indicates an intimate connection between tropical linear algebra and combinatorial optimisation problems, for which many other examples were given by Butkovi\v{c}~\cite{butkovic-combinatorics}. Other important examples of such connection are the metric matrix $F^+$ and the Kleene star $F^*$, defined for $F\in\Rmax^{d\times d}$ in the case $\lambda(F)\leq 0$ as the matrix series
\begin{equation}
\label{e:metricmatrix}    
    F^+=F\oplus F^{\otimes 2}\oplus\ldots\oplus F^{\otimes d},
\end{equation}
\begin{equation}
\label{e:kleenestar}    
    F^*=I\oplus F\oplus\ldots\oplus F^{\otimes d-1}.
\end{equation}
It is easy to see that 1) for arbitrary $t$, any entry $F^{\otimes t}_{ij}$ is the maximum weight of a walk on $\digr(F)$ connecting $i$ to $j$ with length $t$, 2) any entry $F^+_{ij}$ of the metric matrix is the maximum weight of a walk on $\digr(F)$ connecting $i$ to $j$ of arbitrary length.  

Grigoriev and Shpilrain~\cite{grigoriev2014tropical,grigoriev2018tropical} suggested a number of protocols based on the tropical linear algebra, which is briefly introduced above. In particular, they suggested a tropical version of Stickel's protocol, motivated by the lack of genuine  additive inverses and by the lack of multiplicative inverses of generic tropical matrices: a tropical matrix cannot be inverted unless it is a generalised monomial matrix~\cite{butkovic}. However, subsequently, an attack on their tropical implementation of Stickel's protocol was suggested by Kotov and Ushakov~\cite{kotov2015analysis}. Furthermore, in a previous publication~\cite{muanalifahmodifying} 
we analysed a number of other tropical implementations of Stickel's protocol based on commuting matrices in tropical algebra and developed a generalization of the attack~\cite{kotov2015analysis}, which applies to all of them. Although this attack becomes inefficient as the number of monomials or generators of the domain of commuting matrices increases, it is quite successful and motivates the search of other protocols based on tropical algebra. To this end, Grigoriev and Shpilrain~\cite{grigoriev2018tropical} suggested new protocols based on two different versions of tropical semidirect product. An attack on both protocols was more recently suggested by Rudy and Monico~\cite{RM-2005} and then another attack (on Protocol 1) by Isaac and Kahrobei~\cite{IK-2011}. For discussion of these attacks, see Subsection~\ref{ss:discussion} in the end of this paper.

Our first aim here to solve what we call the tropical discrete logarithm problem (see Problem~\ref{prob:discretelog} below), as we think that it can be quite important for the existing and future protocols in tropical cryptography. Thus
we formulate a tropical discrete logarithm problem and suggest a solution of it based on the weak CSR expansion of Merlet et al.~\cite{MNS}.  The solution is also very closely related to the quadratic bound on the ultimate periodicity of critical rows and columns of tropical matrix powers obtained by Nachtigall~\cite{Nacht} and improved by Merlet et al.~\cite{MNSS}. Theoretically, the solution is guaranteed to work in some special cases, but it also has 100\% success in our numerical experiments.

We then show how our solution to the tropical discrete logarithm problem can be applied to suggest yet another attack on Protocol 1 of~\cite{grigoriev2018tropical}. The attack is based on the ultimate periodicity of (the critical columns of) the tropical matrix powers, since, as we show in Proposition~\ref{p:MHpowers}, the semidirect powers used in Protocol 1 of~\cite{grigoriev2018tropical} can be expressed via the tropical matrix powers. 

The rest of the paper is organised as follows. In Section~\ref{s:ultper} we give more background on the ultimate periodicity in tropical linear algebra, formulate the tropical discrete logarithm problem and give a solution to this problem. We then prove that the solution is guaranteed to work in some special cases. In Section~\ref{s:semidirect} we revisit the tropical semidirect product used by Grigoriev and Shpilrain~\cite{grigoriev2018tropical} to construct their Protocol 1. In particular, we show how the messages exchanged by Alice and Bob are related to tropical matrix powers. The protocol and attack on it are described in Section~\ref{s:protocol}. This is followed by some toy examples,  discussion of numerical experiments and attacks suggested by Issac and Kahrobaei~\cite{IK-2011} and Rudy and Monico~\cite{RM-2005}.

\section{Discrete logarithm problem and ultimate periodicity}
\label{s:ultper}

In this section we will discuss the algorithmic solution of the following problem, which we call the tropical discrete logarithm.
\begin{problem}[Tropical Discrete Logarithm]
\label{prob:discretelog}
Suppose that $V\in\Rmax^{m\times d}$, $F\in\Rmax^{d\times d}$ and secret key $t\geq 1$ are used to produce $A=V\otimes F^{\otimes t}$. Knowing $A,$
$V$ and $F$ and that $t$ is unique, find $t$. 
\end{problem}

There is an important special case, in which the tropical discrete logarithm is well defined. 

\begin{lemma}
\label{l:disclog}
Suppose that $V$ has finite entries and $F$ is irreducible. Then $V\otimes F^{\otimes t_1}\neq V\otimes F^{\otimes t_2}$ for any $t_1$ and $t_2$ with $t_1\neq t_2$ if and only if $\lambda(F)\neq 0$.
\end{lemma}
\begin{proof}
Suppose that we have $V\otimes F^{\otimes t_1}= V\otimes F^{\otimes t_2}$ for some  $t_1<t_2$. However, then each row $V\otimes F^{\otimes t_1}$ with some finite entries (which has some finite entries since so does $V$ and since $F$ is irreducible) is a left eigenvector of $F^{\otimes (t_2-t_1)}$ with eigenvalue $0$.
However, since $F$ is irreducible, by \cite[Theorem 4.4.8]{butkovic}  $\lambda(F)\neq 0$ is the only eigenvalue of $F$ (both for left and for right eigenvectors) and by \cite[Corollary 5.5]{TwoCores} the set of eigenvalues of $F^{\otimes (t_2-t_1)}$ consists only of the value $(t_2-t_1)\times \lambda(F)\neq 0$, thus it is the unique eigenvalue of $F^{\otimes (t_2-t_1)}$ (both for left and for right eigenvectors). This contradiction shows that $V\otimes F^{\otimes t_1}=V\otimes F^{\otimes t_2}$ implies $t_1=t_2$, so the tropical discrete logarithm is well-defined.

If $\lambda(F)=0$, then the sequence $(V\otimes F^{\otimes t})_{t\geq 1}$ is ultimately periodic~\cite{CDQV-85,CDQV-83} (see also \cite{baccelli1992synchronization,butkovic,heidergott2014max}), implying that the tropical discrete logarithm is not well-defined.
\end{proof}

Now consider $F\in\Rmax^{d\times d}$ with $\lambda(F)\neq -\infty$. The critical graph of $F$, denoted by $\crit(F)$, is the subgraph of $\digr(F)$, which consists of all nodes and arcs of the cycles where the maximum cycle mean $\lambda(F)$ is attained. It is easy to see that the critical graph in general consists of several strongly connected components (abbreviated as s.c.c.), which do not have any connection to one another.

Cyclicity of each component of $\crit(F)$ is defined as g.c.d. of the lengths of all cycles of that component .
Now, suppose that the critical graph $\crit(F)$ has $l$ s.c.c. $\crit_1, \ldots,\crit_l$
with corresponding cyclicities $\sigma_1,\ldots,\sigma_l$.
For all $\nu\in \{1,...,l\}$, each $\nu$th component gives rise to a $CSR$ term via the following procedure.  
 
Let $\lambda=\lambda(F)$.
Denote $U_{\nu}=((\lambda^-\otimes F)^{\otimes\sigma_{\nu}})^+$ (using the metric matrix defined in~\eqref{e:metricmatrix}). Then, let matrices $C_{\nu},$ $R_{\nu}$ and $S_{\nu}$ be defined by:
\begin{equation}
\label{csrdef}
\begin{split}
(C_{\nu})_{ij}&=
\begin{cases}
(U_{\nu})_{ij} &\text{if $j$ is in $\crit_{\nu}$}\\
\0 &\text{otherwise,}
\end{cases}\quad
(R_{\nu})_{ij}=
\begin{cases}
(U_{\nu})_{ij} &\text{if $i$ is in $\crit_{\nu}$}\\
\0 &\text{otherwise,}
\end{cases}\\
(S_{\nu})_{ij}&=
\begin{cases}
\lambda^- \otimes F_{ij} &\text{if $(i,j)\in \crit_{\nu}$}\\
\0 &\text{otherwise.}
\end{cases}
\end{split}
\end{equation}

Define also matrices $B_{\nu}[F]$ and $B[F]$ by
\begin{equation}
\label{e:bdef}
(B_{\nu}[F])_{ij}=
\begin{cases}
-\infty, & \text{if $i\in\crit_{\nu}$ or $j\in\crit_{\nu}$},\\
F_{ij}, &\text{otherwise},
\end{cases}\quad 
(B[F])_{ij}=
\begin{cases}
-\infty, & \text{if $i\in\crit(F)$ or $j\in\crit(F)$},\\
F_{ij}, &\text{otherwise}.
\end{cases}
\end{equation}

Denote by $t(\rem\sigma)$ the remainder of $t$ modulo $\sigma$ (i.e., $r\in\{0,\ldots,\sigma-1\}$) such that $t=k\sigma+r$ for some $k$. Denote
 $C_{\nu} S_{\nu}^{k} R_{\nu}[F]=
 C_{\nu}\otimes S_{\nu}^{k}\otimes R_{\nu}$ and $C_{\nu} S_{\nu}^{k}[F]=
 C_{\nu}\otimes S_{\nu}^{\otimes k}$ for more brevity and to indicate the matrix ($F$) from which $C_{\nu}$, $S_{\nu}$ and $R_{\nu}$ are defined. 
 
The following claims can be derived from certain results of~\cite{MNS}, see Appendix.

\begin{proposition}[Coro. of \cite{MNS}, Theorem 4.1 and Corollary 4.3]
\label{p:CSR}
Let $F\in\Rmax^{d\times d}$ with $\lambda=\lambda(F)\neq -\infty$ and suppose that $\crit(F)$ has components $\crit_1,\ldots,\crit_l$
and $\sigma_{\nu}$ for $1\leq \nu\leq l$ are their 
cyclicities. Then for any $\nu\in\{1,\ldots,l\}$ 
\begin{equation}
\label{e:CSR}
F^{\otimes t}=\lambda^{\otimes t}\otimes  C_{\nu}S_{\nu}^{t(\rem\sigma_\nu)}R_{\nu}[F] \oplus (B_{\nu}[F])^{\otimes t},\qquad \forall t\geq (d-1)^2+1.     
\end{equation}
\end{proposition}

\begin{proposition}
\label{p:cyclicity}
Under the conditions of Proposition~\ref{p:CSR}, we also have
\begin{equation}
\label{e:CSRdec}
F^{\otimes t}=\lambda^{\otimes t}\otimes \left(\bigoplus_{\nu=1}^l C_{\nu}S_{\nu}^{t(\rem\sigma_\nu)}R_{\nu}[F]\right)\oplus 
(B[F])^{\otimes t}
\qquad \forall t\geq (d-1)^2+1,     
\end{equation}
Furthermore, if $F$ is irreducible then there exists $T(F)$
such that 
\begin{equation}
\label{e:CSRirred}
 F^{\otimes t}=\lambda^{\otimes t}\otimes \left(\bigoplus_{\nu=1}^l C_{\nu}S_{\nu}^{t(\rem\sigma_\nu)}R_{\nu}[F]\right),\qquad\forall t\geq T(F).   
\end{equation}
\end{proposition}

Equation~\eqref{e:CSRirred} implies that after $T(F)$ the sequence of powers 
$(\lambda(F)^-\otimes F)^{\otimes t}$ is periodic, with period equal to the least common multiple of $\sigma_{\nu}$ for $\nu=1,\ldots,l$, a well-known fact established by Cohen et. al.~\cite{CDQV-85,CDQV-83}.

\if{
The above result has the following immediate corollary.
\begin{corollary}
\label{c:premain}
Let $V\in\Rmax^{m\times d}$ and $F\in\Rmax^{d\times d}$ with 
$\lambda=\lambda(F)\neq -\infty$, and let $\crit_1$ be one of the components of $\crit(F)$, treated as the first one in the procedure above. Then for any $t\geq (d-1)^2+1$, the columns of $V\otimes F^{\otimes t}$ with indices in $\crit_1$ are equal to the corresponding columns in 
$\lambda^{\otimes t}\otimes V\otimes C_1S_1^{t(\rem\sigma_1)}[F]$.
\end{corollary}
\begin{proof}
Using~\eqref{e:CSR} we see that the columns of $F^{\otimes t}$ with indices in $\crit_1$ are equal to the corresponding columns of $\lambda^{\otimes t}\otimes C_1S_1^{t(\rem\sigma_1)}R_1[F]$ for any $t\geq (d-1)^2+1$, since the same columns in all other CSR terms are equal to $-\infty$. By~\cite{SerSch}[Corollary 3.7], we can replace the above with $\lambda^{\otimes t}\otimes C_1S_1^{t(\rem\sigma_1)}[F]$. The claim follows as we premultiply the columns of $F^{\otimes t}$ and $\lambda^{\otimes t}\otimes C_1S_1^{t(\rem\sigma_1)}[F]$ by $V$.   
\end{proof}
}\fi

It is not too difficult to compute the CSR terms. In particular, one needs to find $\lambda$, for which one can exploit Karp's method with complexity $O(d^3)$~\cite{baccelli1992synchronization, butkovic} or the policy iteration algorithm of Cochet-Terrasson et al.~\cite{Coc-98,heidergott2014max}, which works in general case and is very efficient in practice. The usual technique for powering up a matrix is to use repeated squaring, and this yields the addition of an $O(d^3\log d)$ term (observing that $\sigma_{\nu}\leq d$).
Further, the metric matrix can be computed by shortest path algorithms such as Floyd-Warshall~\cite{baccelli1992synchronization,butkovic,heidergott2014max}.  The complexity of finding the components of $\crit(F)$ does not exceed $O(d^3)$~\cite{baccelli1992synchronization}. We also need to know the cyclicity of the components, which can be computed in $O(d^2)$ by Balcer and Veinott's digraph condensation~\cite{DigraphCond}. However, below we are going to show how some of these problems can be avoided, as instead of the whole critical component we can use one critical cycle from that component, following an idea of Merlet et al.~\cite[Theorem 6.1]{MNS}. The resulting complexity of computing CSR remains of the order $O(d^3\log d)$, but we avoid the need for identifying the whole components of $\crit(V)$ and the use of Balcer-Veinott digraph condensation.

Let us first give yet another definition of a CSR term, as below.
Suppose that $\cycle$ is a critical cycle, with length 
$l(\cycle)$. Denote $U_{\cycle}=((\lambda^-\otimes F)^{\otimes l(\cycle)})^+$. Then, let matrices $C_{\cycle},$ $R_{\cycle}$ and $S_{\cycle}$ and $B_{\cycle}[F]$ be defined by:
\begin{equation}
\label{csrdefnu}
\begin{split}
(C_{\cycle})_{ij}&=
\begin{cases}
(U_{\cycle})_{ij} &\text{if $j$ is in $\cycle$}\\
\0 &\text{otherwise,}
\end{cases}\quad
(R_{\cycle})_{ij}=
\begin{cases}
(U_{\cycle})_{ij} &\text{if $i$ is in $\cycle$}\\
\0 &\text{otherwise,}
\end{cases}\\
(S_{\cycle})_{ij}&=
\begin{cases}
\lambda^- \otimes F_{ij} &\text{if $(i,j)\in \cycle$}\\
\0 &\text{otherwise,}
\end{cases}
\end{split}
\end{equation}

Proof of the following statement is deferred to Appendix. However, it can be also seen as a corollary of \cite[Theorem 6.1]{MNS}.

\begin{proposition}[Coro. of~\cite{MNS}, Theorem 6.1]
\label{p:cycleCSR}
Let $\cycle$ be a cycle belonging to a component $\crit_{\nu}$ of 
the critical graph of a square matrix $F$ with $\lambda(F)\neq -\infty$. Then $C_{\nu}S_{\nu}^tR_{\nu}[F]=C_{\cycle}S_{\cycle}^tR_{\cycle}[F]$ for any natural $t$, and therefore:
\begin{equation}
\label{e:CSRcycle}
F^{\otimes t}=\lambda^{\otimes t}\otimes  C_{\cycle}S_{\cycle}^{t(\rem l(\cycle))}R_{\cycle}[F] \oplus (B_{\nu}[F])^{\otimes t},\qquad \forall t\geq (d-1)^2+1.     
\end{equation}
for any critical cycle $\cycle$ and component $\crit_{\nu}$
in which it lies.
\end{proposition}

Here, equation~\eqref{e:CSRcycle} follows from~\eqref{e:CSR} and 
the first part of the claim since $\sigma_{\nu}$ divides $l(\cycle)$ and $\{C_{\nu}S_{\nu}^tR_{\nu}[F]\}_{t\geq 0}$ is periodic with period $\sigma_{\nu}$ by~\cite[Prop. 3.2]{SerSch}.

The next immediate corollary of above results will be used in practice, for solving the tropical discrete logarithm problem.
It is closely related to an observation by Nachtigall~\cite{Nacht} that critical rows and columns of matrix powers become periodic after $O(d^2)$, and the further more refined results of Merlet et al.~\cite{MNSS}. 

\begin{corollary}
\label{c:main}
Let $V\in\Rmax^{m\times d}$ and $F\in\Rmax^{d\times d}$ with 
$\lambda=\lambda(F)\neq -\infty$, and let $\cycle$ be a cycle of $\crit(F)$. Then for any $t\geq (d-1)^2+1$, the columns of $V\otimes F^{\otimes t}$ with indices in $\cycle$ are equal to the corresponding columns in $\lambda^{\otimes t}\otimes V\otimes C_{\cycle}S_{\cycle}^{t(\rem l(\cycle))}R_{\cycle}[F]$.
\end{corollary}
\begin{proof}
Equation~\eqref{e:CSRcycle} implies that the columns of $F^{\otimes t}$ with indices in $\cycle$ are equal to the 
corresponding columns of $\lambda^{\otimes t}\otimes C_{\cycle}S_{\cycle}^{t(\rem l(\cycle))}R_{\cycle}[F]$. 
The claim now follows as we premultiply the columns of $F^{\otimes t}$ and $\lambda^{\otimes t}\otimes C_{\cycle}S_{\cycle}^{t(\rem l(\cycle))}R_{\cycle}[F]$ with indices in $\cycle$ by $V$.   
\end{proof}

Corollary~\ref{c:main} suggests the following algorithm 
for finding $t$ such that $A=V\otimes F^{\otimes t}$, that is, for solving Problem~\ref{prob:discretelog}. In this algorithm, $E$ will denote a matrix of appropriate dimensions consisting of all zeros.

\begin{algorithm}[Finding the tropical discrete logarithm]~\\
\label{a:disclog}
{\rm
{\bf Input:} $A,\, V\in\Rmax^{m\times d}$, $F\in\Rmax^{d\times d}$.\\
{\bf Output:} $t$ such that $A=V\otimes F^{\otimes t}$.
\begin{itemize}
    \item[0.] Find $\lambda=\lambda(F)$ and a critical cycle 
    $\cycle$. Compute $C_{\cycle}$ and $S_{\cycle}$ according to~\eqref{csrdefnu}.
    \item[1.] For $t=0,\, 1,\ldots, (d-1)^2$ check if 
    $A=V\otimes F^{\otimes t}$ and return $t$ if it is found;
    \item[2.] For $k=0,\ldots, l-1$, where $l=l(\cycle)$, check if $A_{\cdot i}-V\otimes (C_{\cycle}S_{\cycle}^kR_{\cycle}[F])_{\cdot i}=\mu+ E_{\cdot i}$ for all $i\in\cycle$ and some $\mu$ such that $t=\mu/\lambda(F)$ is a natural number and return the first such $t$ that is found. 
\end{itemize}
}
\end{algorithm}

\begin{proposition}
\label{p:disclog-comp}
Part 0., part 1. and part 2. of Algorithm~\ref{a:disclog} require at most $O(d^3\log l(\cycle))$, $O(md^4)$ and $O(ml(\cycle)(d+l(\cycle)))$ operations, respectively. 
\end{proposition}
\begin{proof}
Complexity bounds:
\begin{itemize}
    \item [0.] Finding $\lambda(F)$ and a critical cycle $\cycle$ needs at most $O(d^3)$ operations (Karp's algorithm and the methods described in~\cite{heidergott2014max,ORE-99}. After this, $C_{\cycle}$ can be found in $O(d^3\log l(\cycle))$ operations (dominated by the repeated matrix squaring).  
    \item [1.]  At step 1, the outer loop has size $(d-1)^2$, and the computationally dominant operation is that of repeated multiplication of an $m\times d$ matrix by an $d\times d$ matrix $F$, taking $md^2$ operations. Thus, the overall complexity is $O(md^4)$.  
    \item [2.] At step 2, the computational complexity can be decreased using the observation that the columns of $C_{\cycle}S_{\cycle}^{t(\rem l(\cycle))}R_{\cycle}[F]$ with indices in $\cycle$ are equal to the corresponding columns 
     $C_{\cycle}S_{\cycle}^{t(\rem l(\cycle))}[F]$ by \cite[Corollary 3.7]{SerSch}, and therefore we actually check 
     if $A_{\cdot i}-(V\otimes C_{\cycle}\otimes S_{\cycle}^{\otimes k})_{\cdot i}=\mu+ E_{\cdot i}$ for all $i\in\cycle$ and some $\mu$ such that $t=\mu/\lambda(F)$ is a natural number (the same for all $i$).
    The outer loop has size $l(\cycle)$ and we precompute the columns of $V\otimes C_{\cycle}$ with indices in $\cycle$, which gives $O(mdl(\cycle))$ operations. The computationally dominant operation at each step is that of multiplying an $m\times l$ matrix by $S_{\cycle}$ (done by a permutation of and adding some scalar values to the columns of that matrix), which is $O(ml(\cycle))$. Overall it gives $O(ml(\cycle)(d+l(\cycle)))$.
\end{itemize}
\end{proof}

\begin{remark}
\label{r:CSR} 
{\rm Using \cite[Corollary 3.7]{SerSch}, $C_{\cycle}S_{\cycle}^{t\rem(l(\cycle))}R_{\cycle}[F]$ can be replaced with $C_{\cycle}S_{\cycle}^{t\rem(l(\cycle))}[F]$ in Corollary~\ref{c:main} and Algorithm~\ref{a:disclog}.} 
\end{remark}

\begin{remark}
\label{r:lighterversion} 
{\rm We can also suggest a lighter but less reliable version of Algorithm~\ref{a:disclog} where  $A_{\cdot i}-(V\otimes C_{\cycle}S_{\cycle}^{t(\rem l(\cycle))}[F])_{\cdot i}=\mu+ E_{\cdot i}$ is checked just for one $i\in\cycle$. Then the complexity of Step 2. drops further.}
\end{remark}

\begin{theorem}
\label{t:disclog} 
Suppose that matrices $V\in\Rmax^{m\times d}$, $F\in\Rmax^{d\times d}$ and critical cycle $\cycle$ are such that 
any of the following equivalent conditions holds:
\begin{itemize}
\item [1.] For any $t_1 \neq t_2$, we have
$V\otimes\lambda^{\otimes t_1}\otimes C_{\cycle}S_{\cycle}^{t_1\rem l(\cycle)}[F]\neq V\otimes\lambda^{\otimes t_2}\otimes  C_{\cycle}S_{\cycle}^{t_2\rem l(\cycle)}[F]$,
\item[2.] For no $t_1,t_2\geq (d-1)^2+1$, $t_1\neq t_2$ we have that all columns of $V\otimes F^{\otimes t_1}$ with indices in $\cycle$
are equal to the corresponding columns of $V\otimes F^{\otimes t_2}$.
\end{itemize}
Then, for any $A=V\otimes F^{\otimes t}$ with $t\geq (d-1)^2+1$, part 2. of Algorithm~\ref{a:disclog} finds this $t$ and it is unique. 
\end{theorem}
\begin{proof}
The equivalence between 1. and 2. follows by
Corollary~\ref{c:main} and Remark~\ref{r:CSR}, which also imply that  if  $t\geq (d-1)^2+1$, then  $A_{\cdot i}=t\times\lambda+ V\otimes (C_{\cycle}S_{\cycle}^{t(\rem l(\cycle)}[F])_{\cdot i}$ and hence for $k=t\rem(l(\cycle)$ we have
 $A_{\cdot i}-V\otimes (C_{\cycle}S_{\cycle}^k[F])_{\cdot i}=\mu+ E_{\cdot i}$ for all $i\in\cycle$, where $\mu$ is such that $t=\mu/\lambda(F)$ is natural. Furthermore, if this holds for $t\geq (d-1)^2+1$, then we have $A_{\cdot i}=\lambda^{\otimes t}\otimes V\otimes (C_{\cycle}S_{\cycle}^{t(\rem l(\cycle))}[F])_{\cdot i}$ for all $i\in\cycle$, and hence $A_{\cdot i}= (V\otimes F^{\otimes t})_{\cdot i}$ for all such $i$ by Proposition~\ref{p:cycleCSR} and~\eqref{e:CSR}. Condition 2. of the theorem then implies that such $t$ is unique and hence correct.
\end{proof}



\begin{remark}
\label{r:badcase}
{\rm The algorithm cannot work when $\lambda(F)=0$. In this case, obviously, the sequence of  columns $\{(V\otimes F^{\otimes t})_{\cdot i}\}_{t>(d-1)^2}$ is
periodic for any $i\in\cycle$ with the same period, and there are infinitely many $t$ such that
$A_{\cdot i}= (V\otimes F^{\otimes t})_{\cdot i}$, if one such $t$ exists. However, if $F$ is irreducible with $\lambda(F)=0$, then the tropical discrete logarithm problem is not well-defined, either.} 
\end{remark}

\if{
\begin{corollary}
\label{c:goodcase1}
Suppose that $V\otimes \lambda^{\otimes t_1}\otimes C_1S_1^{(t_1\rem\sigma_1)}R[F]\neq V\otimes \lambda^{\otimes t_2}\otimes C_1S_1^{(t_2\rem\sigma_1)}R_1[F]$ for any $t_1$ and $t_2$. Then Algorithm~\ref{a:disclog} finds $t$ such that $A=V\otimes F^{\otimes t}$ for any $t$. 
\end{corollary}
}\fi

The following corollary gives a simplification of above conditions in an important special case.

\begin{corollary}
\label{c:goodcase}
Suppose that $V$ has finite entries, $F$ is irreducible, $\crit(F)$ is strongly connected and $\lambda(F)\neq 0$.
Then, if $A=V\otimes F^{\otimes t}$ then Algorithm~\ref{a:disclog} finds this $t$.
\end{corollary}
\begin{proof}
Lemma~\ref{l:disclog} shows that the tropical discrete logarithm is well-defined in this case.

For $t\leq (d-1)^2$, Algorithm checks the equality $A=V\otimes F^{\otimes t}$ in a straightforward way, and there is nothing to prove. Assume that $t\geq (d-1)^2+1$. For the validity of Algorithm, it suffices to show that the condition of this corollary implies the condition of Theorem~\ref{t:disclog}. For this, suppose that by the contrary that condition 1. of Theorem~\ref{t:disclog} is violated. Then we have 
$V\otimes\lambda^{\otimes t_1}\otimes C_{\cycle}S_{\cycle}^{t_1\rem l(\cycle)}[F]=V\otimes\lambda^{\otimes t_2}\otimes  C_{\cycle}S_{\cycle}^{t_2\rem l(\cycle)}[F]$
for some $t_1\neq t_2$.  
Postmultiplying it by $R_{\cycle}$ and 
multiplying it by $\lambda^{\otimes kl(\cycle)}$ for big enough $k$ (if necessary), we obtain  
\begin{equation}
\label{e:CSR-interm0}
V\otimes\lambda^{\otimes t_1}\otimes  C_{\cycle}S_{\cycle}^{t_1(\rem l(\cycle))}R_{\cycle}[F]=V\otimes\lambda^{\otimes t_2}\otimes  C_{\cycle}S_{\cycle}^{t_2(\rem l(\cycle))}R_{\cycle}[F], 
\end{equation}
for some $t_1\neq t_2$, $t_1,t_2\geq T(F)$.
Now let $\sigma$ be the cyclicity of the critical graph. As the critical graph is strongly connected, there is a unique CSR term (with $C$, $S$ and $R$ defined using $\crit_1=\crit(F)$).
Using Proposition~\ref{p:cycleCSR} and that $l(\cycle)$ is a multiple of $\sigma$ we rewrite~\eqref{e:CSR-interm0} as
\begin{equation*}
V\otimes\lambda^{\otimes t_1}\otimes  CS^{t_1(\rem\sigma)}R[F]=
V\otimes\lambda^{\otimes t_2}\otimes  CS^{t_2(\rem\sigma)}R[F], 
\end{equation*}
for some
$t_1\neq t_2$, $t_1,t_2\geq T(F)$.
Now recall that we have $F^{\otimes t}=\lambda^{\otimes t}\otimes CS^{t(\rem\sigma)}R[F]$ for all $t\geq T(F)$ by Proposition~\ref{p:cyclicity}, hence
$V\otimes F^{\otimes t_1}=V\otimes F^{\otimes t_2}$ for some $t_1,t_2\geq T(F)$ and $t_1\neq t_2$, violating the result of Lemma~\ref{l:disclog}. So the condition of this theorem implies any of the equivalent conditions of Theorem~\ref{t:disclog}, and the claim follows. 
\end{proof}

\if{
For the purpose of what follows next, we will also consider the following special case.

\begin{corollary}
\label{c:working} 
Suppose that $F=I\oplus H$, $F$ is irreducible with strongly connected $\crit(F)$ and $\lambda(F)>0$,
and that $V\neq -\infty$. Then for any given $A=V\otimes F^{\otimes t}$, Algorithm~\ref{a:disclog} finds this $t$. 
\end{corollary}
\begin{proof}
We will argue that equality $V\otimes F^{\otimes t_1}=V\otimes F^{\otimes t_2}$ is impossible in this case for $t_1\neq t_2$. On the contrary, suppose that it holds and, say, $t_1<t_2$. Since $F\geq I$, we have $$ V\otimes F^{\otimes t_1}\leq V\otimes F^{\otimes (t_1+1)}\leq \ldots V\otimes F^{\otimes t_2}.$$
As the left and the right ends of this chain of inequalities coincide, we obtain $V\otimes F^{\otimes t_1}=V\otimes F^{\otimes (t_1+1)}$. This implies that $F$ has a left eigenvector with eigenvalue $0$: any row of $V\otimes F^{\otimes t_1}$ that is not identically $-\infty$. But the irreducibility of $F$ implies that the eigenvalue is unique and it is not equal to $0$, a contradiction. 
\end{proof}
}\fi

\section{Semidirect products and powers}
\label{s:semidirect}

Grigoriev and Shpilrain~\cite{grigoriev2018tropical}  consider the following {\em semidirect product} of the pairs of matrices over tropical semiring
\begin{equation}
\label{e:semidirect1}
    (M,G)(A,H)=((M\circ H)\oplus A, G\circ H)
\end{equation}

We will consider one of the protocols in~\cite{grigoriev2018tropical}, where  $\circ$ is defined as the {\em adjoint product:} 
\begin{equation}
\label{e:adjoint}
A\circ B= A\oplus B\oplus A\otimes B,
\end{equation}
defined for any square matrices $A$ and $B$ of the same size.
It has the following properties:
\begin{itemize}
    \item $(A\circ B)\circ C=A\circ (B\circ C)$ (associativity),
    \item $A\circ (B\oplus C)=A\circ B\oplus A\circ C$ and $(B\oplus C)\circ A=B\circ C\oplus B\circ A$
    (distributivity).
\end{itemize}

Adjoint product~\eqref{e:semidirect2} can be used to define {\em adjoint powers} inductively: $A^{\circ(k+1)}=A^{\circ k}\circ A$ for all $k$. Moreover, the associativity implies that
for any nonzero numbers
$m_1,\ldots,m_s\in\Nat$ such that $m_1+\ldots+m_s=k$ we have 
\begin{equation}
\label{e:adjointpowerprop}
A^{\circ k}=A^{\circ m_1}\circ A^{\circ m_2}\circ \ldots  \circ A^{\circ m_s}.     
\end{equation}
Thus the adjoint powers $A^{\circ n}=\underbrace{A\circ\ldots\circ A}_n$ are well-defined and can be quickly computed using~\eqref{e:adjointpowerprop}. 
Alternatively, the following identity for them can be offered:
\begin{equation}
\label{e:Acircn}
A^{\circ n}=A\oplus A^{\otimes 2}\oplus \ldots \oplus A^{\otimes n}.    
\end{equation}
Indeed, $A^{\circ 2}=A\oplus A^{\otimes 2}$ is obvious, and for general $n$ we can use a simple 
induction:
\begin{equation*}
\begin{split}
A^{\circ n}&=A^{\circ(n-1)}\circ A=A\oplus A^{\circ(n-1)}\oplus (A^{\circ (n-1)}\otimes A)\\
&= A\oplus (A\oplus A^{\otimes 2}\oplus\ldots\oplus A^{\otimes(n-1)})\oplus (A^{\otimes 2}\oplus\ldots \oplus A^{\otimes n})\\
&= A\oplus A^{\otimes 2}\oplus\ldots\oplus A^{\otimes n}.
\end{split}
\end{equation*}

Using~\eqref{e:Acircn} we also observe the following:
\begin{proposition}
Let $A\in\Rmax^{d\times d}$ have $\lambda(A)\leq 0$ and $n\geq d$. Then $A^{\circ n}=A^+$. 
\end{proposition}
Here $A^+$ is the {\em metric matrix} of $A$ defined in~\eqref{e:metricmatrix}.

With $\circ$ being the adjoint multiplication, the semidirect product of $(M,G)$ and $(A,H)$ given by~\eqref{e:semidirect1} becomes
\begin{equation}
\label{e:semidirect2}
(M,G)(A,H)=(M\oplus A\oplus H\oplus M\otimes H, G\otimes H\oplus G\oplus H). 
\end{equation}

The semidirect product is associative: we have
\begin{equation}
\label{e:assoc}
[(M,G)\cdot(A,H)]\cdot (B,J)=(M,G)\cdot [(A,H)\cdot (B,J)]
\end{equation}
For the proof of this property, see Appendix.

Semidirect product~\eqref{e:semidirect2} can be used to define semidirect powers of matrix pairs inductively: $(M,H)^{k+1}=(M,H)^k\cdot (M,H)$
for all $k$. Moreover, associativity~\eqref{e:assoc} implies that
for any nonzero numbers
$m_1,\ldots,m_s\in\Nat$ such that $m_1+\ldots+m_s=k$ we have \begin{equation}
\label{e:powerprop}
(M,H)^k=(M,H)^{m_1}(M,H)^{m_2}\dots (M,H)^{m_s}.     
\end{equation}
This property assures that the semidirect powers $(M,H)^{k}=\underbrace{(M,H)\cdot\ldots\cdot (M,H)}_k$ are well-defined. We now express the semidirect powers in terms of the tropical matrix powers.

\begin{proposition}
\label{p:MHpowers}
Let $M,H\in\Rmax^{d\times d}$. Then
\begin{equation*}
(M,H)^k=
((M\otimes\bigoplus_{i=0}^{k-1} H^{\otimes i}) \oplus (H\otimes\bigoplus_{i=0}^{k-2} H^{\otimes i}), H^{\circ k}) \end{equation*}
for all $k\geq 2$.
\end{proposition}
\begin{proof}
We first consider $k=2$ to check the base of induction. 
We obtain:
\begin{equation*}
    (M,H)(M,H)=(M\oplus H\oplus M\oplus M\otimes H,\ H^{\circ 2})=(M\otimes (I\oplus H)\oplus H,\ H^{\circ 2}). 
\end{equation*}
We now assume that the statement holds for $k=t$ and 
prove it for $k=t+1$. Indeed:
\begin{equation*}
\begin{split}
& (M,H)^{t+1}=(M,H)^t\cdot (M,H)= ((M\otimes\bigoplus_{i=0}^{t-1} H^{\otimes i}) \oplus (H\otimes\bigoplus_{i=0}^{t-2} H^{\otimes i}), H^{\circ t})
  \cdot (M,H)\\
  &=((M\otimes\bigoplus_{i=0}^{t-1} H^{\otimes i}) \oplus (H\otimes\bigoplus_{i=0}^{t-2} H^{\otimes i})\oplus 
  M\oplus H\oplus (M\otimes\bigoplus_{i=1}^{t} H^{\otimes i}) \oplus (H\otimes\bigoplus_{i=1}^{t-1} H^{\otimes i}), H^{\circ (t+1)})\\
  &=((M\otimes\bigoplus_{i=0}^{t} H^{\otimes i}) \oplus 
  (H\otimes\bigoplus_{i=0}^{t-1} H^{\otimes i}), H^{\circ(t+1)}).
 \end{split}    
\end{equation*}
The induction is complete.
\end{proof}

Note that we can also use that $\bigoplus_{i=0}^k H^{\otimes i}= (I\oplus H)^{\otimes k}$ 
for any $k$, and then the result of the previous proposition can be reformulated as follows:
\begin{equation}
\label{e:MHpower-regroup}
\begin{split}
(M,H)^k&=(M\otimes (I\oplus H)^{\otimes (k-1)}\oplus H\otimes (I\oplus H)^{\otimes (k-2)}, H^{\circ k})\\
       &=((M\otimes (I\oplus H)\oplus H)\otimes (I\oplus H)^{\otimes (k-2)}, H^{\circ k}).
\end{split}
\end{equation}

\section{The protocol and its cryptanalysis}
\label{s:protocol}

In this section we will use the following order relations between matrices $A,\, B\in\Rmax^{m\times n}$ of same dimensions. We write: 
\begin{itemize}
    \item $A\leq B$ (resp. $A\geq B$), if $a_{ij}\leq b_{ij}$ (resp. $a_{ij}\geq b_{ij}$) for all $i\in\{1,\ldots,m\}$ and $j\in\{1,\ldots n\}$;
    \item $A<B$, if $A\leq B$ and $A\neq B$;
    \item $A>B$, if $A\geq B$ and $A\neq B$.
    \end{itemize}

\subsection{The protocol under question}
Based on the property~\eqref{e:powerprop}, Grigoriev and Shpilrain~\cite{grigoriev2018tropical} suggested the following protocol using tropical semidirect powers: 

\begin{protocol}[Grigoriev and Shpilrain~\cite{grigoriev2018tropical}]~\\
\label{prot:main}
\begin{itemize}
    \item[1.] Alice and Bob agree on public matrices $M,H\in\Zmax^{d\times d}$ (that is, $d\times d$ matrices $M$ and $H$ whose entries are integer numbers or $-\infty$);
    \item[2.] Alice selects a private positive integer $m$ and Bob selects 
    a private positive integer $n$;
    \item[3.] Alice computes $(M,H)^m=(A,H^{\circ m})$ and sends $A$ to Bob;
    \item[4.] Bob computes $(M,H)^n=(B,H^{\circ n})$ and sends $B$ to Alice;
    \item[5.] Alice computes $K_a=A\oplus B\oplus H^{\circ m}\oplus (B\otimes H^{\circ m})$;
    \item[6.] Bob computes $K_b=A\oplus B\oplus H^{\circ n}\oplus (A\otimes H^{\circ n})$.
\end{itemize}
\end{protocol}

Property~\eqref{e:powerprop} implies that $K_a=K_b$, since both of them are the first component of $(M,H)^{m+n}$.

For the protocol recalled above, 
we immediately obtain using~\eqref{e:MHpower-regroup}
\begin{equation}
\label{e:AB}
\begin{split}
A&=\left(M\otimes\bigoplus_{i=0}^{m-1} H^{\otimes i}\right) \oplus \left(H\otimes \bigoplus_{i=0}^{m-2} H^{\otimes i}\right)=(M\otimes (I\oplus H)\oplus H)\otimes (I\oplus H)^{\otimes (m-2)},\\
B&=\left(M\otimes \bigoplus_{i=0}^{n-1} H^{\otimes i}\right) \oplus \left(H\otimes\bigoplus_{i=0}^{n-2} H^{\otimes i}\right)=(M\otimes (I\oplus H)\oplus H)\otimes (I\oplus H)^{\otimes (n-2)},
\end{split}
\end{equation}
for the messages exchanged between Alice and Bob ($m\geq 2$ and $n\geq 2$), using Proposition~\ref{p:MHpowers}.

We have the following immediate corollary of these expressions.

\begin{corollary}
\label{c:A<B}
We have the following implications:
\begin{itemize}
    \item[{\rm (i)}] $m>n\Rightarrow A\geq B$, $n>m\Rightarrow B\geq A$;
    \item[{\rm (ii)}] $A>B\Rightarrow m>n,$ $B>A\Rightarrow n>m$.
\end{itemize}
\end{corollary}

In the next sections we describe the attack on the Grigoriev-Shpilrain protocol, which depends on the sign of $\lambda(H)$.

Let us denote 
\begin{equation*}
V=M\otimes (I\oplus H)\oplus H.
\end{equation*}
The messages sent by Alice and Bob can be expressed as
\begin{equation*}
A=V\otimes (I\oplus H)^{\otimes (m-2)},\quad B=V\otimes (I\oplus H)^{\otimes (n-2)},    
\end{equation*}  
as it follows from~\eqref{e:AB}. Hence, under the assumptions of Theorem~\ref{t:disclog} or Corollary~\ref{c:goodcase}, we can apply Algorithm~\ref{a:disclog} to $A,B,V$ and $F=I\oplus H$ to find $m-2$ and $n-2$ (unless $m=1$ or $n=1$). Notice, however, that this algorithm cannot be applied when $\lambda(H)\leq 0$, and this motivates a separate treatment of this case.

\subsection{Case $\lambda(H)\leq 0$}
\label{ss:easycase}
Recall that if $\lambda(H)\leq 0$ for $H\in\Rmax^{d\times d},$ then we have
\begin{equation}
H^*=I\oplus H\oplus \ldots H^{\otimes (d-1)}    
\end{equation}
We then immediately obtain the following corollary of~\eqref{e:AB}.
\begin{corollary}
Let $M,H\in\Rmax^{d\times d}$ and $\lambda(H)\leq 0$. If $m\geq d+1$ then $A=(M\oplus H)\otimes H^*,$ and if $n\geq d+1$ then 
$B=(M\oplus H)\otimes H^*$.
\end{corollary}

Using this corollary and~\eqref{e:Acircn}, if $m\geq d+1$ or if $A=(M\oplus H)\otimes H^*$ we also obtain
\begin{equation*}
\begin{split}
A\otimes H^{\circ n}&=(M\oplus H)\otimes (H^*\otimes H^{\circ n})\\
&=(M\oplus H)\otimes \left(H^*\otimes \bigoplus_{i=1}^n H^{\otimes i} \right)\leq (M\oplus H)\otimes H^*=A,\\
H^{\circ n}&=\bigoplus_{i=1}^n H^{\otimes i} \leq A, 
\end{split}
\end{equation*}
and also $B\leq A$, using~\eqref{e:AB}. If $n\geq d+1$ or if $B=(M\oplus H)\otimes H^*$ then we have
\begin{equation*}
\begin{split}
B\otimes H^{\circ m}&=(M\oplus H)\otimes (H^*\otimes H^{\circ m})\\
&= (M\oplus H)\otimes \left(H^*\otimes \bigoplus_{i=1}^m H^{\otimes i} \right)\leq  (M\oplus H)\otimes H^*=B,\\
H^{\circ m}&=\bigoplus_{i=1}^m H^{\otimes i}\leq B
\end{split}
\end{equation*}
and $A\leq B$.
Therefore, we have
\begin{equation*}
\begin{split}
K_a&= B\oplus A\oplus H^{\circ m}\oplus B\otimes H^{\circ m}=A\oplus B=B, \quad \text{if $n\geq d+1$}\\
K_b&= A\oplus B\oplus H^{\circ n}\oplus A\otimes H^{\circ n}=A\oplus B=A, \quad \text{if $m\geq d+1$}.
\end{split}    
\end{equation*}

Thus we arrive at the following result.
\begin{proposition}
Let $M,H\in\Rmax^{d\times d}$ and $\lambda(H)\leq 0$ and let $m\geq d+1$, $n\geq d+1$, $A=(M\oplus H)\otimes H^*$ or $B=(M\oplus H)\otimes H^*$. Then
\begin{equation*}
K_a=K_b=A\oplus B=(M\oplus H)\otimes H^*.    
\end{equation*}
\end{proposition}

Thus in this case the key can be computed simply as $A\oplus B$.

\subsection{Computing the key with known $m$ and $n$}

If we have $m$ and $n$ then the key can be obviously computed as 
 \begin{equation}
\label{e:keycompute}
    K_a=K_b=A\oplus B\oplus H^{\circ m}\oplus (B\otimes H^{\circ m})=
    A\oplus B\oplus H^{\circ n}\oplus (A\otimes H^{\circ n}),
\end{equation}
where $H^{\circ m}$ and $H^{\circ n}$ can be computed as adjoint powers, using~\eqref{e:adjointpowerprop} or \eqref{e:Acircn}.

Let us also consider how to simplify expression~\eqref{e:keycompute}. Assume first that $m>n$. Then $A\geq B$ and $A\geq H^{\circ n}$, since any power $H^{\otimes i}$ 
for $1\leq i\leq n$ appears as one of the terms in 
$$
A=(M\otimes (I\oplus H)\oplus H)(I\oplus H\oplus\ldots \oplus H^{\otimes (m-2)}),
$$
when we multiply it out. Then the key simplifies to
\begin{equation}
\label{e:Kacompute}
K_a=K_b=A\otimes(I\oplus H^{\circ n})=A\otimes (I\oplus H\oplus\ldots \oplus H^{\otimes n})=A\otimes (I\oplus H)^{\otimes n}.
\end{equation}

In the case $n>m$ we similarly obtain 
\begin{equation}
\label{e:Kbcompute}
K_a=K_b=B\otimes(I\oplus H\oplus\ldots \oplus H^{\otimes m})=B\otimes (I\oplus H)^{\otimes m}.
\end{equation}

In the case $m=n$ we have $B=A$ and therefore 
\begin{equation}
\label{e:Km=n}
K_a=K_b= A\otimes (I\oplus H)^{\otimes n}\oplus H\otimes (I\oplus H)^{\otimes (n-1)}.
\end{equation}



\subsection{Attacking the protocol}

Let us now give a more formal description of the attack on Protocol~\ref{prot:main}, in the form of an algorithm.

\begin{algorithm}[Attacking Protocol~\ref{prot:main}]~\\
\label{a:attack}
{\rm 
{\bf Input:} public matrices $M,H\in\Zmax^{d\times d}$ and messages $A,B\in\Zmax^{d\times d}$ of Alice and Bob.\\
{\bf Output:} common key $K_a=K_b$.
\begin{itemize}
    \item[0.] Compute $\lambda(H),$ $F=I\oplus H$ and 
    $V=(M\otimes (I\oplus H)\oplus H)$.
    \item[1.] If $\lambda(H)\leq 0$ then check if 
    $A=(M\oplus H)\otimes H^*$ or $B=(M\oplus H)\otimes H^*$. If any of these two conditions is true then return $K=(M\oplus H)\otimes H^*$.\\
    If none of these conditions are true, check if $A=M$ or $B=M$ or find $l_1,l_2=0,\ldots d-2$ such that $A=V\otimes F^{\otimes l_1}$ and $B=V\otimes F^{l_2}$. Then set $m=l_1+2$ or $m=1$ if $A=M$, and $n=l_2+2$ or $n=1$ if $B=M$,
    and go to 3.
    \item[2.] If $\lambda(H)>0$ then check $A=M$ or $B=M$ or find $l_1$ and $l_2$ satisfying $A=V\otimes F^{\otimes l_1}$ and $B=V\otimes F^{\otimes l_2}$ using Algorithm~\ref{a:disclog}.
    Then set $m=l_1+2$ or $m=1$ if $A=M$, and $n=l_2+2$ or $n=1$ if $B=M$, and go to 3.
    \item[3.] Compute the key using~\eqref{e:Kacompute},~\eqref{e:Kbcompute} or~\eqref{e:Km=n}.  
\end{itemize}
}
\end{algorithm}

The increasing property of $F=I\oplus H$ means that the sequence of matrices 
$\{M,\; V,\; V\otimes F,\; V\otimes F^2\ldots \}$ is non-decreasing, and it either stabilises so that 
$V\otimes F^{\otimes t}= (M\oplus H)\otimes H^*$ for $t\geq T$
for some $T\leq d-1$, or it grows in such a way that 
$$M<V< V\otimes F^{\otimes t_1}<V\otimes F^{\otimes t_2}<\ldots$$
In particular, we have $V\otimes F^{\otimes t_1}\neq V\otimes F^{\otimes t_2}$ for $t_1\neq t_2$,  
unless both are equal to $(M\oplus H)\otimes H^*$. These observations, together with the validity of Algorithm~\ref{a:disclog}, imply the following claim, where by $\mathcal{E}$ we define a matrix of arbitrary dimensions with all entries equal to $-\infty$.

\begin{proposition}
 Suppose that $F$ is irreducible with strongly connected $\crit(F)$, and that $V=(M\otimes (I\oplus H)\oplus H)\neq \mathcal{E}$. Then 
the attacker can compute the key using Algorithm~\ref{a:attack}.
\end{proposition}

Let us analyse how many operations the algorithm requires.

\begin{itemize}
\item[0.] Computation of $\lambda(H)$ and $V$ requires no more than 
$O(d^3)$ operations.
\item[1.] Checking if $A=(M\oplus H)\otimes H^*$ or $B=(M\oplus H)\otimes H^*$ requires $O(d^3)$ operations. Straightforward checking for powers less than $d-1$ requires $O(d^4)$ operations.
\item[2.] Here we apply Algorithm~\ref{a:disclog}, whose complexity is analysed in Proposition~\ref{p:disclog-comp}. However, see also the discussion in Subsection~\ref{ss:discussion}.
\item[3.] Computation of the key (unless it has been computed at step 1) requires no more than $O(d^3\log\min(m,n))$. This is  done using repeated tropical matrix squaring.  
\end{itemize}


\section{Examples, numerical experiments and discussion}
\subsection{Toy examples}

We first give a couple of toy examples to demonstrate how the attack on the protocol works in the cases $\lambda(H)\leq 0$ and $\lambda(H)>0$.

\begin{example}[$(\lambda(H)\leq 0)$]
\label{eq:h<0}
{\rm 
Let 
$$M =\begin{pmatrix} 8&7 & 2\\
10&3 &6 \\
-10& -1&3\end{pmatrix},\quad H =\begin{pmatrix} 0& -3&-5 \\
-1&-2 &2 \\
1&-3 &-4\end{pmatrix}.$$ 
Bob and Alice pick two random integer numbers $m=5$ and $n=8$ respectively. Alice and Bob compute 
$$
A=B=
\begin{pmatrix} 10& 7&9\\10& 7&9\\4&1 &3\end{pmatrix}=
K_a =K_{b}.
$$
Since $\lambda(H)=0$, we cannot use tropical discrete logarithm method to find $m$ and $n$. However, Eve can check that $A=B=(M\oplus H)\otimes H^*$, hence she concludes that $K_a=K_b=(M\oplus H)\otimes H^*$.}
\end{example}
\begin{example}[$(\lambda(H)>0)$]
{\rm Alice and Bob agree on public matrices 
$$M =
\begin{pmatrix}
-75& -45& -69&60 \\
83&52 &9 &-72 \\
27&92 &92 &-16 \\
87&93 &-3&84
\end{pmatrix},\quad H =
\begin{pmatrix}
1&7 &2 &5 \\
-1&-2 &2 &4 \\
3& 4&2 &2 \\
-5&-10 &10 &0 
\end{pmatrix}.$$ Then they follow the protocol as follows
\begin{itemize}
\item Alice and Bob pick two random integer numbers $m=15$ and $n=16$ respectively.
\item Alice computes $(M,H)^m= (A,H^{\circ m})$ and Bob computes $(M,H)^n=(B,H^{\circ n})$. They exchange the following messages:
$$A =
\begin{pmatrix}
145& 146&148 &144\\
176&177 &179 &175 \\
175&176 &178 &174\\
176&177 &179 &175
\end{pmatrix},\quad  B=\begin{pmatrix}
151&152 &154 &150\\
182&183 &185 &181 \\
181&182 &184 &180\\
182&183 &185 &181
\end{pmatrix}.$$
\item Alice computes $K_{a}=A\oplus B\oplus H^{\circ m}\oplus (B\otimes H^{\circ m})$
and $K_{b}=B\oplus A\oplus H^{\circ n}\oplus (A\otimes H^{\circ n})$. They thus obtain the common secret key:
$$
K_a=K_b=
\begin{pmatrix}
241& 242&244 &240\\
272&273 &275 &271 \\
271&272 &274 &270\\
272&273 &275 &271
\end{pmatrix}.$$
\end{itemize}

\noindent \textbf{Attacking the protocol}\\
Eve as an attacker only knows public matrices $M$ and $H$ and public keys $A$ and $B$. To attack the protocol Eve needs to find $m$ and $n$ and compute $K_{a}$ or $K_{b}$. Using Algorithm 5.3, Eve obtains Alice's private key by the following:
\begin{enumerate}
    \item Eve computes $\lambda(H)=6$ and $$F=I\oplus H= \begin{pmatrix}
1&7 &2 &5 \\
-1&0 &2 &4 \\
3& 4&2 &2 \\
-5&-10 &10 &0 
\end{pmatrix},\quad  
V =
\begin{pmatrix}
55&50 &70 &60 \\
98&99 &97 &97 \\
95& 96& 94&96 \\
92&93 &95 &97 
\end{pmatrix}.$$
\item Since $\lambda(H)>0$, Eve needs to find $m_{a}$ satisfying $A=V\otimes F^{\otimes (m_{a}-2)}.$ For this Eve finds a critical cycle $Z = (1\;2\;4\;3)$ and computes
$$C_Z=R_Z=
    \begin{pmatrix}
    0& 1 &3 &-1 \\
    -5& 0&2 &-2 \\
    -1& -2& 0& -4\\
    1&2 &4 & 0
    \end{pmatrix},\quad  S_Z=\begin{pmatrix}
    -\infty & 1 &-\infty & -\infty\\
    -\infty& -\infty &-\infty & -2  \\
    -3& -\infty & -\infty & -\infty\\
    -\infty& -\infty& 4  & -\infty
    \end{pmatrix}. $$
\item The dimension is $d=4$, hence for  $t=0,\ldots,(4-1)^2=9$, Eve first tries to find $t$ such that $A=V\otimes F^{\otimes t}$. Here we cannot find $t$ satisfying $A=V\otimes F^{\otimes t}$ for these low exponents. 
\item Now Eve uses the CSR method. The length of critical cycle is $l=4$, but it turns out that 
$$C_ZS_Z^kR_Z[F]
= \begin{pmatrix}
    0& 1 &3 &-1 \\
    -1& 0&2 &-2 \\
    -3& -2& 0& -4\\
    1&2 &4 & 0
    \end{pmatrix}\quad 
    \text{for all $k$}.
$$
For $k=0$ Eve finds that $A=V\otimes (C_ZR_Z[F])=\mu +E$ with $\mu=78$. Eve then finds that  $m_{a}=\mu/\lambda(F)+2=\frac{78}{6}+2=15$.
\item Eve computes $K_a=B\otimes (I\otimes H)^{\otimes 15} = \begin{pmatrix}
241& 242&244 &240\\
272&273 &275 &271 \\
271&272 &274 &270\\
272&273 &275 &271
\end{pmatrix}$.
\end{enumerate}
}
\end{example}

\subsection{Numerical experiments}
\label{ss:numexp}
In this section we will describe the numerical experiments which we performed with the tropical discrete logarithm and attack on \cite[Protocol 1]{grigoriev2018tropical}.  

We first discuss how we generated matrix $F$, which gets powered up in the discrete logarithm problem, or matrix $H$ for \cite[Protocol 1]{grigoriev2018tropical}. 
If we generate matrix $F$ by random and all of its entries are real, then it will be irreducible and generically we will have only one critical cycle. 
This case is the same as the one described in Corollary~\ref{c:goodcase}, in which our solution of the tropical discrete logarithm problem and our attack on Protocol~\ref{prot:main} (\cite[Protocol 1]{grigoriev2018tropical})
are guaranteed to work. Therefore, in part of our experiments, we generate matrices $F$ (and $H$) in such a way that the critical graph is guaranteed to have at least three components.

In more detail, we are doing it as follows:
\begin{itemize}
    \item [(a)] We determine two random integer numbers $k_1$ and $k_2$, where $k_1$ is approximately $\frac{1}{3}$ of the dimension of matrix $d$ and $k_2$ is a random integer number between $k_1$ and $d$. Then we generate three random matrices with entries $0$ and $-\infty$. Each matrix has dimension $k_1$, $k_2-k_1$ and $d-k_2$ respectively. The frequency of $0$ entries is approximately $\frac{1}{3}$ and we make sure that each of these matrices contains a cycle and there is $-\infty$ on the diagonal.
    \item [(b)] We compose a $d\times d$ matrix with entries 
    in $\{0,-\infty\}$, which has the three matrices generated above as as its principal submatrices. The rest of entries in this matrix are set to $-\infty$.
    \item [(c)] We substitute all $-\infty$ entries in step (b) with a random negative number in the interval $[-100,-1]$ and add to the whole matrix a nonzero random number $\lambda\in [1,100]$.
    \item [(d)] We apply a diagonal similarity scaling $A\mapsto D^{-1}\otimes A\otimes D$ where $D$ is a diagonal matrix with all diagonal entries equal to $d_i$ for $i\in\{1,\ldots,n\}$ and all off-diagonal entries equal to $-\infty$. We write  $D=\operatorname{diag}(d_1,\ldots,d_n)$ for such matrix, and the max-algebraic inverse of it can be written as $D^{-1}=\operatorname{diag}(-d_1,\ldots, -d_n)$. In our case, the diagonal entries $d_i$ are randomly  selected in the interval $[-100,100]$.
\end{itemize}

The resulting matrix can be then used as matrix $H$ in~\cite[Protocol 1]{grigoriev2018tropical}, however here we also need to make sure that $\lambda(H)>0$, otherwise we are in the very easy case, treated in Subsection~\ref{ss:easycase}. Note that this is guaranteed by taking $\lambda>0$ at step c), as this is the maximum cycle mean of the matrix generated following (a), (b), (c) and (d). 

For the tropical discrete logarithm problem as well as for the protocol, we run similar experiments using the following parameters:
\begin{itemize}
    \item Dimension $d$ is in the interval $[6,500]$;
    \item The entries of matrix $M$ are random integer numbers in the interval $[-100,100]$; 
    \item Exponents $m,n$ used by Alice and Bob, and the secret key $t$ in the tropical discrete logarithm are random integer numbers in the interval $[(d-1)^2+1,d^2]$.
\end{itemize}
We coded all our attacks in MATLAB and performed experiments using MATLAB R2019/b, also using supercomputer Bluebear system (University of Birmingham) for dimensions between 400 and 500. 
We run 100 experiments for each dimension $d$:
\begin{enumerate}
    \item[1.] We solved the tropical discrete logarithm by Algorithm~\ref{a:disclog} where we skipped step (1): straightforward ``catching''  powers up to $(d-1)^2$. In this experiment we found $100\%$ success rate.
    \item[2.] We attacked \cite[Protocol 1]{grigoriev2018tropical} using Algorithm~\ref{a:attack}. In this experiment we also found $100\%$ success rate.
\end{enumerate}
For the dimensions up to $100$, the average computation times are given on Figure~\ref{fig:comptimes}. We distinguish between the cases where $H$ is randomly generated and where $F=I\oplus H$ is guaranteed to have three critical components. However, the average time that it takes is similar (being slightly less for the case of special matrices), and it does not exceed 6 seconds for dimensions up to $100$ in both cases.
\begin{figure}
    \centering
    \includegraphics[scale=0.5]{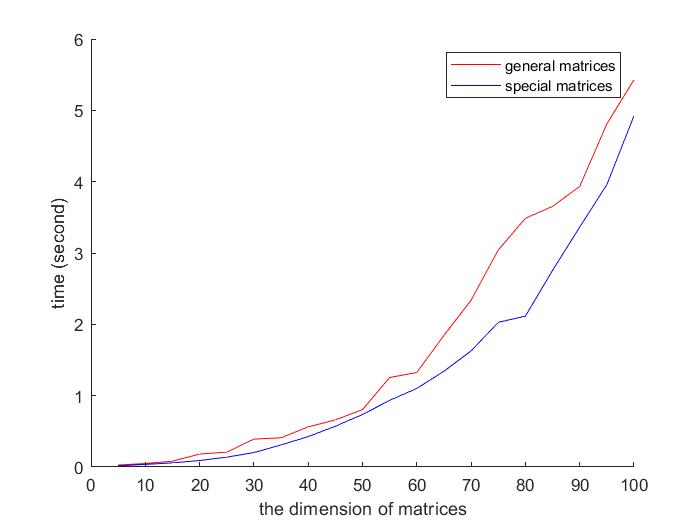} 
    \caption{Time required by Algorithm~\ref{a:attack} in the case where $H$ is randomly generated (``general matrices'') and in the case where $\crit(F)$ is guaranteed to have at least three critical components and $\lambda(F)>0$ (``special matrices)''}
    \label{fig:comptimes}
\end{figure}

\subsection{Discussion}
\label{ss:discussion}

To our knowledge, the first attack on~\cite[Protocol 1]{grigoriev2018tropical} was suggested by Rudy and Monico~\cite{RM-2005}. The attack is based on the property that the sequence $(A^{(k)})_{k\geq 1}$, where $A^{(k)}$ is defined by $(M,H)^k=(A^{(k)},H^{\circ k})$, is nondecreasing (if viewed in max-plus algebra). This allows Eve to apply a binary search to find the secret keys of Alice and Bob. This attack is guaranteed to work and can be efficiently implemented~\cite{RM-2005}. However, its worst-case computational complexity is $O(K^2)$, where $K$ is an upper bound on the logarithm of the secret keys of Alice and Bob. 

Isaac and Kahrobaei~\cite{IK-2011} take a different approach. They find the secret keys of Alice and Bob based on the assumption that the sequence $(A^{(k)})_{k\geq 1}$ is ultimately periodic, which means that $A_{ij}^{(k+p)}=\xi+A_{ij}^{(k)}$ for all $k\geq D$, all indices $i$ and $j$ , some real $\xi$ and some $D,\,p\geq 1$. Not being dependent on the magnitude of the secret keys of Alice and Bob, this attack is more efficient in practice. In view of Proposition~\ref{p:MHpowers} and~\eqref{e:MHpower-regroup} of the present paper, the ultimate periodicity assumption holds when $H$ is irreducible (i.e., when graph associated with $H$ is strongly connected). However, it generally fails when Alice and Bob choose $H$ to be reducible, in which case the attack of~\cite{IK-2011} would not apply. Another implicit assumption for the efficiency of this attack is that the defect (i.e., the periodicity transient) $D$ is rather small. But the magnitude of $D$ can be arbitrarily large for a sequence of tropical matrix powers, which can pose a problem when Alice and Bob are allowed choose $M$ and $H$ (even if with real entries only) to force a large $D$. 

Clearly, our attack on Protocol~\ref{prot:main}, which is based on the ultimate periodicity of the critical columns of $\left((I\oplus H)^{\otimes k}\right)_{k\geq 1}$ (and, therefore, the same columns of $A^{(k)}$) can be applied in the general case. As in~\cite{IK-2011} and unlike the attack of~\cite{RM-2005}, the computation of secret keys of Alice and Bob in our attack does not depend on the magnitude of these secret keys. Furthermore, our attack is directly based on the solution of the tropical discrete logarithm problem in the case $\lambda(H)>0$ and is reduced to the optimal paths problem (i.e., computation of the Kleene star) for $\lambda(H)\leq 0$. 

Although the statement of the tropical discrete logarithm problem is quite obvious, the authors are unaware of works in tropical algebra literature, where this problem was posed and solved. As for Algorithm~\ref{a:disclog}, which we are suggesting, there is clearly some room for improvement. Firstly, the most inefficient part of this algorithm is the straightforward ``catching powers'' up to $(d-1)^2$ in part 1., and here we see the potential in using the ideas of \cite{RM-2005} and~\cite{IK-2011}. In the case of $F=I\oplus H$, where the tropical matrix powers are nondecreasing, we can use the binary search as in~\cite{RM-2005}. Note that this decreases the complexity of Step 1 to $O((m+d)d^2(\log d)^2)$ (similar to the estimate of~\cite{RM-2005} but with $(\log d)^2$ instead of $K^2$). This holds for the application of Algorithm~\ref{a:disclog} to our attack. Secondly, as noticed in Remark~\ref{r:lighterversion}, we could try to check $A_{\cdot i}-(V\otimes C_{\cycle}S_{\cycle}^{t(\rem l(\cycle))}[F])_{\cdot i}=\mu+ E_{\cdot i}$ just for one $i\in\cycle$.  Thirdly, our theoretical claims can guarantee that Algorithm~\ref{a:disclog} works in some special cases, which includes the generic case encountered when $F$ has randomly chosen real entries. However, we have not found any counterexample to Algorithm~\ref{a:disclog} in the case where the critical graph has several components, indicating that its performance and hence the performance of Algorithm~\ref{a:attack} could be guaranteed in a more general case. Such counterexamples, as well as more refined and more efficient versions of Algorithm~\ref{a:disclog} and Algorithm~\ref{a:attack}, guaranteed in more general cases, will be sought in the future.

\section*{Acknowledgement}
We would like to thank the anonymous referee of our paper for their careful reading, useful comments and appreciation of our work.

\appendix
\section{Associativity of the semidirect product}

In this appendix we show that the semidirect product, which we are considering, is indeed associative. We need to prove:
\begin{equation*}
[(M,G)\cdot(A,H)]\cdot (B,J)=(M,G)\cdot [(A,H)\cdot (B,J)]
\end{equation*}
Indeed, on the left-hand side we have:
\begin{align*}
    [(M,G)\cdot (A,H)]\cdot (B,J)&= (M\circ H\oplus A,\ G\circ H)\cdot (B,J)\\
                &= ((M\oplus H\oplus (M\otimes H)\oplus A)\circ J\oplus B,\ G\circ H\circ J)\\
                &= (M\circ J\oplus H\circ J\oplus (M\otimes H)\circ J\oplus A\circ J\oplus B, G\circ H\circ J)\\
                &= (M\oplus J \oplus (M\otimes J)\oplus H\oplus J\oplus (H\otimes J) \oplus 
                (M\otimes H)\oplus J\\
                &~~\oplus (M\otimes H\otimes J) \oplus A \oplus J \oplus (A\otimes J) \oplus B, G\circ H\circ J)\\
                &=(M\oplus H\oplus J\oplus A\oplus B\oplus (M\otimes J)\oplus (M\otimes H)\oplus (H\otimes J) \oplus (A\otimes J)\\
                &\oplus (M\otimes H \otimes J),\ G\circ H\circ J)
\end{align*}
On the right-hand side:
\begin{align*}
    (M,G)\cdot [(A,H)\cdot (B,J)]&= (M,G)\cdot (A\circ J\oplus B, H\circ J)\\
    &= (M,G)\cdot (A\oplus J\oplus A\otimes J\oplus B, H\oplus J\oplus H\otimes J)\\
    &=(M\circ(H\oplus J\oplus H\otimes J)\oplus A\oplus J\oplus A\otimes J\oplus B, G\circ H\circ J)\\
    &=(M\circ H\oplus M\circ J\oplus M\circ (H\otimes J)\oplus A\oplus J\oplus A\otimes J \oplus B, G\circ H\circ J)\\
    &= (M\oplus H\oplus (M\otimes H)\oplus M\oplus J\oplus (M\otimes J)\oplus M\oplus (H\otimes J)\oplus\\ 
    &(M\otimes H\otimes J)\oplus 
    A\oplus J\oplus (A\otimes J) \oplus B, G\circ H\circ J)\\
    &=(M\oplus H\oplus J\oplus A\oplus B\oplus (M\otimes J)\oplus (M\otimes H)\oplus (H\otimes J)\\ 
    &\oplus (A\otimes J)\oplus (M\otimes H \otimes J) ,G\circ H\circ J),
\end{align*}
which is identical with what we obtained for the left-hand side.

\section{CSR proofs}


\subsection{Proofs of Proposition~\ref{p:CSR} and Proposition~\ref{p:cyclicity}}

Here we deduce these propositions from results of Merlet et al.~\cite{MNS}. To do this, we need to introduce other versions of CSR decomposition and expansion, which appeared in that work. First of all, we can define the ``big'' CSR terms by considering the whole critical graph instead of individual components. For this, let $\sigma$ be the l.c.m. of all $\sigma_1,\ldots,\sigma_l$ and define $U=((\lambda^-\otimes F)^{\otimes \sigma})^+$. Then let matrices $C,$ $R$ and $S$ be defined by
\begin{equation*}
\begin{split}
C_{ij}&=
\begin{cases}
U_{ij} &\text{if $j$ is in $\crit(F)$}\\
\0 &\text{otherwise,}
\end{cases}\quad
R_{ij}=
\begin{cases}
U_{ij} &\text{if $i$ is in $\crit(F)$}\\
\0 &\text{otherwise,}
\end{cases}\\
S_{ij}&=
\begin{cases}
\lambda^- \otimes F_{ij} &\text{if $(i,j)\in \crit(F)$}\\
\0 &\text{otherwise.}
\end{cases}
\end{split}
\end{equation*}

We will denote $CS^tR[F]=C\otimes S^{\otimes t}\otimes R$.
By Wielandt's bound (15) in \cite[Theorem 4.1]{MNS}, we have
\begin{equation}
\label{e:weakCSR}    
F^{\otimes t}=\lambda^{\otimes t}\otimes CS^tR[F]\oplus (B[F])^{\otimes t}=  
\lambda^{\otimes t}\otimes CS^{t(\rem\sigma)}R[F]\oplus (B[F])^{\otimes t},\quad t\geq (d-1)^2+1    
\end{equation}

Let us now discuss how the CSR term appearing in~\eqref{e:weakCSR} can be decomposed into smaller CSR terms. For this, assume some numbering of the critical components and 
for $\mu\colon 1\leq \mu\leq l-1$, define matrix $F_{\mu+1}$ by
\begin{equation*}
(F_{\mu+1})_{ij}=
\begin{cases}
-\infty, & \text{if $i\in\crit_{\mu}$ or $j\in\crit_{\mu}$},\\
(F_{\mu})_{ij}, &\text{otherwise},
\end{cases}
\end{equation*}
with $F_1=F$. Observe that $\lambda(F_{\mu})=\lambda$ for any such $\mu$, and that the critical graph of $F_{\mu}$ consists of components $\crit_{\mu},\ldots,\crit_l$. Denote $U'_{\mu}=((\lambda^-\otimes F_{\mu})^{\otimes\sigma_\mu})^+$. Then, let matrices $C'_{\mu},$ $R'_{\mu}$ and $S'_{\mu}$
for $\mu=1,\ldots,l$ be defined by:
\begin{equation*}
\begin{split}
(C'_{\mu})_{ij}&=
\begin{cases}
(U'_{\mu})_{ij} &\text{if $j$ is in $\crit_\nu$}\\
\0 &\text{otherwise,}
\end{cases}\quad
(R'_{\mu})_{ij}=
\begin{cases}
(U'_{\mu})_{ij} &\text{if $i$ is in $\crit_\nu$}\\
\0 &\text{otherwise,}
\end{cases}\\
(S'_{\mu})_{ij}&=(S_{\mu})_{ij}=
\begin{cases}
\lambda^- \otimes (F_{\mu})_{ij} &\text{if $(i,j)\in \crit_\nu$}\\
\0 &\text{otherwise.}
\end{cases}
\end{split}
\end{equation*}
Let us also compare $C'_{\mu}$, $S'_{\mu}$ and $R'_{\mu}$ with the matrices introduced in~\eqref{csrdef}. Notice that $S'_{\mu}=S_{\mu},$ for all $\mu$ and also 
$C'_1=C_1$ and $R'_1=R_1$,
but 
in general only $C'_{\mu}\leq C_{\mu}$ and $R'_{\mu}\leq R_{\mu}$. We further 
denote $C'_{\mu}S_{\mu}^{t(\rem\sigma_{\nu})}R'_{\mu}[F]=
C'_{\mu}\otimes S_{\mu}^{t(\rem\sigma_{\mu})}\otimes R'_{\mu}$, similarly to the CSR notation before. According to~\cite[Corollary 4.3]{MNS}, the following decomposition holds:
\begin{equation}
\label{e:CSRdecomp1}
CS^{t(\rem\sigma)}R[F]= \bigoplus_{\mu=1}^l C'_{\mu}S_{\mu}^{t(\rem\sigma_{\mu})}R'_{\mu}[F],\qquad\forall t. \end{equation}

Combining~\eqref{e:weakCSR} and~\eqref{e:CSRdecomp1}, we obtain
\begin{equation}
\label{e:CSRcomb1}
F^{\otimes t}=
\lambda^{\otimes t}\otimes\bigoplus_{\mu=1}^l C'_{\mu}S_{\mu}^{t(\rem\sigma_{\mu})}R'_{\mu}[F]\oplus (B[F])^{\otimes t},\quad \forall t\geq (d-1)^2+1.
\end{equation}
Observing that $C'_1S_1^{t(\rem\sigma_1)}R'_1[F]=C_1S_1^{t\rem\sigma_1}R_1[F]$ we can also write:
\begin{equation}
\label{e:CSRcomb15}
F^{\otimes t}= \lambda^{\otimes t}\otimes C_1S_1^{t(\rem\sigma_1)}R_1[F]\oplus
\lambda^{\otimes t}\otimes \bigoplus_{\mu=2}^l C'_{\mu}S_{\mu}^{t(\rem\sigma_{\mu})}R'_{\mu}[F]\oplus (B[F])^{\otimes t},\quad \forall t\geq (d-1)^2+1.
\end{equation}

But by a similar combination of~\cite{MNS} Theorem 4.1 and Corollary 4.3, we also have:
\begin{equation}
\label{e:CSRcomb2}
    (B_1[F])^{\otimes t}= \lambda^{\otimes t}\otimes \bigoplus_{\mu=2}^l C'_{\mu}S_{\mu}^{t(\rem\sigma_{\mu})}R'_{\mu}[F]\oplus (B[F])^{\otimes t},\quad \forall t\geq (d-1)^2+1,
\end{equation}
where $B_1[F]$ is defined as in~\eqref{e:bdef} with $\nu=1$.
Substituting~\eqref{e:CSRcomb2} into~\eqref{e:CSRcomb15} we obtain 
\begin{equation*}
 F^{\otimes t}= \lambda^{\otimes t}\otimes C_1S_1^{t(\rem\sigma_1)}R_1[F]\oplus (B_1[F])^{\otimes t},    
\end{equation*}
which is the same as~\eqref{e:CSR} for $\nu=1$, thus establishing Proposition~\ref{p:CSR}. 

To explain~\eqref{e:CSRdec} (Proposition~\ref{p:cyclicity}), observe that~\eqref{e:CSRcomb1} holds for any numbering of critical components. In other words, for any numbering of critical components we get the corresponding CSR decomposition of the form~\eqref{e:CSRcomb1}. Depending on which of these components is the first one, the first term in~\eqref{e:CSRcomb1} can be equal to any of the terms $C_{\nu}S_{\nu}^{t(\rem\sigma_{\nu})}R_{\nu}[F]$, while any other term in~\eqref{e:CSRcomb1} is less than or equal to one of these $C_{\nu}S_{\nu}^{t(\rem\sigma_{\nu})}R_{\nu}[F]$. This implies that taking the tropical sum of all CSR decompositions~\eqref{e:CSRcomb1} written for all possible numberings of the critical components we obtain~\eqref{e:CSRdec}:
\begin{equation}
\label{e:CSRdec1}
F^{\otimes t}=
\lambda^{\otimes t}\otimes\bigoplus_{\nu=1}^l C_{\nu}S_{\nu}^{t(\rem\sigma_{\nu})}R_{\nu}[F]\oplus (B[F])^{\otimes t},\quad \forall t\geq (d-1)^2+1,
\end{equation}

For the irreducible matrices, the existence of $T(F)$ such that 
\begin{equation}
\label{e:CSRult}
F^{\otimes t}=\lambda^{\otimes t}\otimes CS^{t(\rem\sigma)}R[F],\qquad \forall t\geq T(F)    
\end{equation}
follows from~\cite[Theorem 5.6]{SerSch}, and a number of upper bounds on $T(F)$ have been established in~\cite{MNS}. Recall also that 
\begin{equation}
 \label{e:CSRineqs1}
 CS^tR[F]=\bigoplus_{\mu=1}^l C'_{\mu}S_{\mu}^{t(\rem\sigma_{\mu})}R'_{\mu}[F]\leq \bigoplus_{\nu=1}^l C_{\nu}S_{\nu}^{t(\rem\sigma_{\nu})}R_{\nu}[F],\quad\forall t.  \end{equation}

It follows from~\eqref{e:CSRdec1} and \eqref{e:CSRult} that
\begin{equation}
\label{e:CSRineqs2}
    \bigoplus_{\nu=1}^l C_{\nu}S_{\nu}^{t(\rem\sigma_{\nu})}R_{\nu}[F]\leq CS^{t(\rem\sigma)}R[F],\quad \forall t\geq\max(T(F),\, (d-1)^2+1).
\end{equation}
Combining~\eqref{e:CSRineqs1},~\eqref{e:CSRineqs2} and the periodicity of CSR terms, we can replace inequalities in~\eqref{e:CSRineqs1} and~\eqref{e:CSRineqs2} with equalities, and we can write~\eqref{e:CSRult} as
$$
F^{\otimes t}=
\lambda^{\otimes t}\otimes\bigoplus_{\nu=1}^l C_{\nu}S_{\nu}^{t(\rem\sigma_{\nu})}R_{\nu}[F],
$$
establishing~\eqref{e:CSRirred} and completing the proof of Proposition~\ref{p:cyclicity}.

\subsection {Proof of Proposition~\ref{p:cycleCSR}}

The proof given below is a 
simplified version of the proof of~\cite[Theorem 6.1]{MNS}

In the beginning of this proof let us introduce some extra notation, following~\cite{MNS}. For walk $W$ denote by $p(W)$ its weight, and for a set of walks $\mathcal{W}$, denote by $p(\mathcal{W})$ the maximal weight of a walk in $\mathcal{W}$. Below we are going to use the following sets of walks:
\begin{itemize}
    \item $\walks{i}{j}:$ set of walks connecting node $i$ to node $j$;
    \item $\walkslen{i}{j}{t}:$ set of walks connecting node $i$ to node $j$ and having length $t$; 
    \item $\walkslen{i}{j}{t,l}:$ set of walks connecting node $i$ to node $j$ and having length $t(\rem l)$;
    \item $\walkslennode{i}{j}{t,l}{\subcrit}$: set of walks connecting node $i$ to node $j$, going through a node in subgraph $\subcrit$ and having length $t(\rem l)$.
\end{itemize}

In particular, we have the following optimal walk interpretation of the entries of a matrix power and a metric matrix~\cite{butkovic}:
\begin{equation}
(A^{\otimes t})_{ij}= p(\walkslen{i}{j}{t}),
\qquad 
(A^+)_{ij}=p(\walks{i}{j}).    
\end{equation}

The proof given below is a simplified version of the 
proof of \cite[Theorem 6.1]{MNS}. For the sake of this proof we assume without loss of generality that the critical graph is strongly connected, i.e., it consists of one component, and let $C,$ $S$ and $R$ be defined from it.
We will show that for arbitrary $i$ and $j$
\begin{equation}
\label{e:goal}
(C\otimes S^{\otimes t}\otimes R)_{ij}=(C_{\cycle}\otimes  S_{\cycle}^{\otimes t}\otimes R_{\cycle})_{ij}=
p(\walkslennode{i}{j}{t,l(\cycle)}{\cycle}).
\end{equation}

We first show
\begin{equation}
\label{e:goal1}
(C_{\cycle}\otimes  S_{\cycle}^{\otimes t}\otimes R_{\cycle})_{ij}
\leq p(\walkslennode{i}{j}{t,l(\cycle)}{\cycle}), \quad
 (C\otimes S^{\otimes t}\otimes R)_{ij}\leq  p(\walkslennode{i}{j}{t,l(\cycle)}{\cycle}).  
\end{equation}
For the first inequality, we 
have $(C_{\cycle}\otimes S_{\cycle}^{\otimes t}\otimes R_{\cycle})_{ij}=
(C_{\cycle})_{is_1}\otimes (S_{\cycle}^{\otimes t})_{s_1s_2}\otimes (R_{\cycle})_{s_2j}$, for some $s_1,s_2\in\cycle$, which means that in terms of walks,
there is a walk $V$ such that $p(V)=(C_{\cycle}\otimes S_{\cycle}^{\otimes t}\otimes R_{\cycle})_{ij}$ and decomposed as $V=V_1V_2V_3$, where 
$V_1\in\walkslen{i}{s_1}{0,l(\cycle)}$, $V_2\in\walkslen{s_1}{s_2}{t}$ and $V_3\in\walkslen{s_2}{j}{0,l(\cycle)}$. It is then obvious that 
$V\in \walkslennode{i}{j}{t,l(\cycle)}{\cycle}$, and the first inequality of~\eqref{e:goal1} follows.

As for the second inequality, $(C\otimes S^{\otimes t}\otimes R)_{ij}$ is the weight of a walk $W$ that can be decomposed as 
$W=W_1W_2W_3$, where 
$W_1\in\walkslen{i}{k_1}{0,\sigma}$, $W_2\in\walkslen{k_1}{k_2}{t}$ and $W_3\in\walkslen{k_2}{j}{0,\sigma}$ and $k_1,k_2\in\crit_1$.
We now introduce a walk 
$W_4$ connecting $k_1$ to a node $k_3\in\cycle$ on $\crit(F)$ and a walk $W_5$ going back to $k_1$ on $\crit(F)$. The composition $W_4W_5$ forms a closed walk on $\crit(F)$, and its length is a multiple of $\sigma$. In $k_3$, we insert a closed walk $W_6$ of a big enough length, whose all arcs belong to $\crit(F)$ and whose length is such that the sum of lengths of $W_1,$ $W_3$, $W_4$, $W_5$ and $W_6$ is a multiple of $l(\cycle)$. Then for the walk $\Tilde{W}=W_1W_4W_6W_5W_2W_3,$ we have $\Tilde{W}\in\walkslennode{i}{j}{t,l(\cycle)}{\cycle}$. 
We thus have $p(W)=p(\Tilde{W})\leq p(\walkslennode{i}{j}{t,l(\cycle)}{\cycle}),$ hence the second 
inequality of~\eqref{e:goal1}.

We now prove:
\begin{equation}
\label{e:goal2}
(C_{\cycle}\otimes  S_{\cycle}^{\otimes t}\otimes R_{\cycle})_{ij}
\geq p(\walkslennode{i}{j}{t,l(\cycle)}{\cycle}), \quad
 (C\otimes S^{\otimes t}\otimes R)_{ij}\geq  p(\walkslennode{i}{j}{t,l(\cycle)}{\cycle}).  
\end{equation}

For this, consider a walk $W$ such that $p(W)=p(\walkslennode{i}{j}{t,l(\cycle)}{\cycle})$. Then we
decompose it as $W=V_1V_2$, where $V_1$ connects $i$
to a node $k\in\cycle\subseteq \crit(F)$, and $V_2$ connects 
$k$ to $j$. At node $k$ we insert $m\cycle$: a number of copies of $\cycle$ such that $m l(\cycle)\geq t+l(\cycle)$. We then find $V_3,$  $W_2$ and $V_4$ such that $m\cycle=V_3W_2V_4,$  $W_2$ has length $t$ and both 
$l(V_1)+l(V_3)$ and $l(V_4)+l(V_2)$ are multiples of $l(\cycle)$.
Since $\Tilde{W}=V_1V_3W_2V_4V_2\in \walkslennode{i}{j}{t,l(\cycle)}{\cycle}$ and $m$ is big enough,
such walks $V_3,$ $W_2$ and $V_4$ can be found. Denoting by $k_1$ the end of walk $V_3$ and by $k_2$ the beginning of walk $V_4$, we see that 
\begin{equation*}
\begin{split}
& p(V_1)+p(V_3)\leq (C_{\cycle})_{ik_1},\quad  
p(V_1)+p(V_3)\leq C_{ik_1},\quad 
p(W_2)\leq (S_{\cycle}^{\otimes t})_{k_1k_2},\quad 
p(W_2)\leq S^{\otimes t}_{k_1k_2},\\
& p(V_4)+p(V_2)\leq (R_{\cycle})_{k_2j},\quad  
p(V_4)+p(V_2)\leq R_{k_2j},
\end{split}
\end{equation*}
and this implies both inequalities of~\eqref{e:goal2}.

\end{document}